\documentclass[draft]{amsart}
\usepackage{amsmath}
\usepackage{amssymb}

\newtheorem{thm}{Theorem}
\newtheorem{cor}[thm]{Corollary}

\newtheorem{lem}[thm]{Lemma}

\theoremstyle{remark}

\numberwithin{thm}{section}
\numberwithin{equation}{section}

\newcommand{\cv}{\mathbb{C}}

\newcommand{\bv}{\mathbb{B}}
\newcommand{\oc}{\mathcal{O}}
\newcommand{\arctanh}{\textup{arctanh}}
\newcommand{\aut}{\textup{Aut}}
\newcommand{\ds}{\displaystyle}

\begin{document}

\title[Fridman invariants and generalized squeezing functions]{On Fridman invariants and generalized squeezing functions}

\author[F. Rong, S. Yang]{Feng Rong, Shichao Yang}

\address{School of Mathematical Sciences, Shanghai Jiao Tong University, 800 Dong Chuan Road, Shanghai, 200240, P.R. China}
\email{frong@sjtu.edu.cn}

\address{School of Mathematics and Statistics, Huizhou University, 46 Yan Da Boulevard, Huizhou, Guangdong, 516007, P.R. China}
\email{yangshichao68@163.com}



\thanks{The authors are partially supported by the National Natural Science Foundation of China (grant no. 11871333).}

\begin{abstract}
In this paper, we introduce the notion of \textit{generalized squeezing function} and study the basic properties of generalized squeezing functions and Fridman invariants. We also study the comparison of these two invariants, in terms of the so-called \textit{quotient invariant}.
\end{abstract}

\maketitle

%\today

\section{Introduction}

Due to the lack of a Riemann mapping theorem in several complex variables, it is of fundamental importance to study the biholomorphic equivalence of various domains in $\cv^n$, $n\ge 2$. For such a study, it is necessary to introduce different kinds of holomorphic invariants. In this paper, we study two such invariants, the Fridman invariants and (generalized) squeezing functions.

The Fridman invariant was defined by Fridman in \cite{Fridman1983} for Kobayashi hyperbolic domains $D$ in $\cv^n$, $n\ge 1$, as follows. Denote by $B_D^k(z,r)$ the $k_D$-ball in $D$ centered at $z\in D$ with radius $r>0$, where $k_D$ is the Kobayashi distance on $D$. For two domains $D_1$ and $D_2$ in $\cv^n$, denote by $\oc_u(D_1,D_2)$ the set of \textit{injective} holomorphic maps from $D_1$ into $D_2$.

Recall that a domain $\Omega\subset \cv^n$ is said to be \textit{homogeneous} if the automorphism group of $\Omega$ is transitive. For any bounded homogeneous domain $\Omega$, set
$$h_D^\Omega(z)=\inf \{1/r:\ B_D^k(z,r)\subset f(\Omega),\ f\in \oc_u(\Omega,D)\}.$$
For comparison purposes, we call $e_D^\Omega(z):=\tanh (h_D^\Omega(z))^{-1}$ the \textit{Fridman invariant} (cf. \cite{Deng-Zhang2019, Nikolov-Verma2018}).

For any bounded domain $D\subset \cv^n$, the \textit{squeezing function} was introduced in \cite{Deng2012} by Deng, Guan and Zhang as follows:
$$s_D(z)=\sup \{r:\ r\bv^n\subset f(D),\ f\in \oc_u(D,\bv^n),\ f(z)=0\}.$$
Here $\bv^n$ denotes the unit ball in $\cv^n$. Comparing with the Fridman invariant, it seems natural to consider more general squeezing functions, replacing $\bv^n$ by other ``model domains".

Recall that a domain $\Omega$ is said to be \textit{balanced} if for any $z\in \Omega$, $\lambda z\in \Omega$ for all $|\lambda|\le 1$. Let $\Omega$ be a bounded, balanced and convex domain in $\cv^n$. The \textit{Minkowski function} $\rho_\Omega$ is defined as (see e.g. \cite{Pflug2013})
$$\rho_\Omega(z)=\inf \{t>0:\ z/t\in \Omega\},\ \ \ z\in \cv^n.$$
Note that $\Omega=\{z\in \cv^n:\ \rho_\Omega(z)<1\}$. Set $\Omega(r)=\{z\in \cv^n:\ \rho_\Omega(z)<r\}$, $0<r<1$. Then for any bounded domain $D\subset \cv^n$, we can define the \textit{generalized squeezing function} as follows:
$$s_D^\Omega(z)=\sup \{r:\ \Omega(r)\subset f(D),\ f\in \oc_u(D,\Omega),\ f(z)=0\}.$$

It is clear from the definitions that both Fridman invariants and generalized squeezing functions are invariant under biholomorphisms, and both take value in $(0,1]$. There have been much study on the Fridman invariant and the squeezing function in recent years, and we refer the readers to two recent survey articles \cite{Zhang2017,Dengsurvey2019} and the references therein for various aspects of the current research on this subject.

The main purpose of this paper is to study some basic properties of both Fridman invariants and generalized squeezing functions. Moreover, we will also discuss the comparison of these two invariants, for which we introduce the \textit{quotient invariant}
$$m_D^\Omega(z)=s_D^\Omega(z)/e_D^\Omega(z),$$
where $D\subset \cv^n$ is bounded and $\Omega\subset \cv^n$ is bounded, balanced, homogeneous and convex.

In section \ref{S:Fridman}, we study basic properties of Fridman invariants, in particular refining several results from Fridman's original work \cite{Fridman1983}. In section \ref{S:squeezing}, we study basic properties of generalized squeezing functions, in particular extending various properties of the squeezing function given in \cite{Deng2012, Deng2016} to the more general setting. In section \ref{S:comparison}, we study the comparison of Fridman invariants and generalized squeezing functions, in particular generalizing previous results from \cite{Nikolov-Verma2018, RY:comparison}.

\section{Fridman invariants}\label{S:Fridman}

Throughout this section, we suppose that $D$ is a Kobayashi hyperbolic domain in $\cv^n$ and $\Omega$ is a bounded homogeneous domain in $\cv^n$ (unless otherwise stated).

We say that $f\in \oc_u(\Omega,D)$ is an \textit{extremal map} at $z\in D$ if
$$B_D^k(z,\arctanh(e_D^\Omega(z)))\subset f(\Omega).$$
It is not known from Fridman's original work \cite{Fridman1983} whether extremal maps exist. However, if we assume $D$ to be bounded or taut, then extremal maps do exist.

\begin{thm}\label{tee}
If $D$ is bounded or taut, then an extremal map exists at each $z\in D$.
\end{thm}

For the proof of Theorem \ref{tee}, we need two lemmas. The first is probably well-known, and we provide a short proof for completeness.

\begin{lem}\label{lbc}
For any domain $D$ in $\cv^n$, $B_D^k(z,r)$ is a subdomain of $D$.
\end{lem}
\begin{proof}
Since the Kobayashi pseudodistance is continuous, $B_D^k(z,r)$ is an open subset of $D$. Since the Kobayashi pseudodistance is inner, for any $w\in B_D^k(z,r)$ there exists a piece-wise $C^1$-curve $s:[0,1]\rightarrow D$ such that
$s(0)=z$, $s(1)=w$ and $\int_0^1 g_D^k(s(t);s'(t))<r$, where $g_D^k$ denotes the Kobayashi pseudometric. This implies that $s([0,1])\subset B_D^k(z,r)$. Hence $B_D^k(z,r)$ is a subdomain of $D$.
\end{proof}

The next lemma is known as the generalized Hurwitz's theorem in several complex variables (see e.g. \cite{shijihuai1996}).

\begin{lem}\label{lh}
Let $D$ be a domain in $\cv^n$ and $\{f_i(z)\}$ be a sequence of injective holomorphic maps from $D$ to $\cv^n$. Suppose that $f_i$'s converge to a map $f: D \rightarrow \cv^n$ uniformly on compact subsets of $D$. Then either $f$ is an injective holomorphic map or $\det f'(z)\equiv 0$.
\end{lem}
\begin{proof}[Proof of Theorem \ref{tee}]
Since the proof for the taut case is similar as (and simpler than) for the bounded case, we will assume that $D$ is bounded.

Without loss of generality, assume that $0\in \Omega$. By definition, there exist a sequence of holomorphic embeddings $f_i:\Omega \rightarrow D$ with $f_i(0)=z$, and a sequence of increasing positive numbers $r_i$ convergent to $\arctanh(e_D^\Omega(z))$ such that $B_D^k(z,r_i) \subset f_i(\Omega)$. Since $D$ is bounded, by Montel's theorem, there exists a subsequence $\{f_{k_i}\}$ of $\{f_i\}$ which converges to a holomorphic map $f:\Omega \rightarrow \bar{D}$ uniformly on compact subsets of $\Omega$.

By Lemma \ref{lbc}, $B_D^k(z,r_i)$'s are increasing subdomains of $D$. Denote $g_i:=f_i^{-1}|B_D^k(z,r_i)$. Since $\Omega$ is bounded, by Montel's theorem, there exists a subsequence $\{g_{k_i'}\}$ of $\{g_{k_i}\}$ which converges to
a holomorphic map $g:B_D^k(z,\arctanh(e_D^\Omega(z)))\rightarrow \bar\Omega$ uniformly on compact subsets of $B_D^k(z,\arctanh(e_D^\Omega(z)))$.

Take $s>0$ such that $\bv^n(z,s)\subset B_D^k(z,r_1)$. By Cauchy's inequality, for any $i$, $|\det g_i'(z)|<c$ for some positive constant $c$. So we have $|\det f_i'(0)|>\frac{1}{c}$, for any $i$. Thus, we have $|\det f'(0)|>0$ and $|\det g'(z)|>0$. By Lemma \ref{lh}, both $f$ and $g$ are injective. In particular, $f(\Omega)\subset D$ and $g(B_D^k(z,\arctanh(e_D^\Omega(z))))\subset \Omega$. Since $f\circ g(w)=w$, for all $w\in B_D^k(z,\arctanh(e_D^\Omega(z)))$, it shows that $f$ is the desired extremal map.
\end{proof}

Based on Theorem \ref{tee}, we can give another proof of \cite[Theorem 1.3(2)]{Fridman1983} as follows.

\begin{thm}
If there exists $z\in D$ such that $e_D^\Omega(z)=1$, then $D$ is biholomorphically equivalent to $\Omega$.
\end{thm}
\begin{proof}
Since $\Omega$ is homogeneous, $s_\Omega(z)\equiv c$ for some positive number $c$. Thus, by \cite[Theorem 4.7]{Deng2012}, $\Omega$ is Kobayashi complete, hence taut.

Without loss of generality, assume that $0\in \Omega$. Let $f_i$'s and $g_i$'s be as in the proof of Theorem \ref{tee}. Since $e_D^\Omega(z)=1$, we have $\bigcup_i B_D^k(z,r_i)=D$.

Since $\Omega$  is taut, by \cite[Theorem 5.1.5]{Kobayashi98}, there exists a subsequence $\{g_{k_i}\}$ of $\{g_i\}$ which converges to a holomorphic map $g:D \rightarrow \Omega$ uniformly on compact subsets of $D$. By the decreasing property of the Kobayashi distance, for $z_1,z_2\in D$ such that $g(z_1)=g(z_2)$, we have for $k_i$ large enough,
$$k_D(z_1,z_2)\le  k_{f_{k_i}(\Omega)}(f_{k_i}\circ g_{k_i}(z_1),f_{k_i}\circ g_{k_i}(z_2))=k_\Omega(g_{k_i}(z_1),g_{k_i}(z_2)).$$
Letting $k_i\rightarrow \infty$, by the continuity of the Kobayashi distance, we have $k_D(z_1,z_2)\le k_\Omega(g(z_1),g(z_2))=0$. Since $D$ is Kobayashi hyperbolic, we have $z_1=z_2$. Thus, $g$ is injective and $D$ is biholomorphic to a bounded domain.

Now Theorem \ref{tee} applies and shows that $D$ is biholomorphically equivalent to $\Omega$.
\end{proof}

It was shown in \cite[Theorem 1.3(1)]{Fridman1983} that $h_D^\Omega(z)$, hence $e_D^\Omega(z)$, is continuous. For its proof, Fridman showed that for $z_1$ and $z_2$ sufficiently close, $|1/h_D^\Omega(z_1)-1/h_D^\Omega(z_2)|\le k_D(z_1,z_2)$. Our next result gives a ``global" version of this estimate in terms of $e_D^\Omega(z)$.

\begin{thm}\label{tec}
For any $z_1$ and $z_2$ in $D$, we have
$$|e_D^\Omega(z_1)-e_D^\Omega(z_2)|\le \tanh[k_D(z_1,z_2)].$$
\end{thm}

For the proof of Theorem \ref{tec}, we need the following basic fact, whose proof we provide for completeness.

\begin{lem}\label{ltanh}
Suppose that $t_i\ge 0$, $i=1,2,3$, and $t_3\le t_1+t_2$. Then,
$$\tanh(t_3)\le \tanh(t_1)+\tanh(t_2).$$
\end{lem}
\begin{proof}
Since $t_3\le t_1+t_2$, we have
$$-\frac{2}{e^{2t_3}+1}-1\le -\frac{2}{e^{2(t_1+t_2)}+1}-1.$$
Define
$$f(t_1,t_2)=\frac{2}{e^{2t_1}+1}+\frac{2}{e^{2t_2}+1}-\frac{2}{e^{2(t_1+t_2)}+1}-1.$$
To show that $\tanh(t_3)\le \tanh(t_1)+\tanh(t_2)$, it suffices to show that $f(t_1,t_2)\le 0$, for all $t_1,t_2\ge 0$. For any fixed $t_1\ge 0$, consider
$$g(t_2)=\frac{2}{e^{2(t_1+t_2)}+1}-\frac{2}{e^{2t_2}+1}.$$
Then,
$$g'(t_2)=-\frac{4e^{2(t_1+t_2)}}{(e^{2(t_1+t_2)}+1)^{2}}+\frac{4e^{2t_2}}{(e^{2t_2}+1)^{2}}.$$
Since the function $\ds \frac{e^t}{(e^t+1)^{2}}$ is decreasing for $t\ge 0$, we have $g'(t_2)\ge 0$ for all $t_2\ge 0$. Hence, $g(t_2)\ge g(0)$ for all $t_2\ge 0$, which implies that $f(t_1,t_2)=g(0)-g(t_2)\le 0$ for all $t_1,t_2\ge 0$.
\end{proof}

\begin{proof}[Proof of Theorem \ref{tec}]
Fix $0<\epsilon<e_D^\Omega(z_1)$, by definition there exists a holomorphic embedding $f:\Omega \rightarrow D$ such that $B_D^k(z_1,\arctanh[e_D^\Omega(z_1)-\epsilon]) \subset f(\Omega)$.

If $z_2\not\in B_D^k(z_1,\arctanh[e_D^\Omega(z_1)-\epsilon])$, then obviously
$$e_D^\Omega(z_2)>0\ge e_D^\Omega(z_1)-\epsilon-\tanh[k_D(z_1,z_2)].$$

If $z_2\in B_D^k(z_1,\arctanh[e_D^\Omega(z_1)-\epsilon])$, then by Lemma \ref{ltanh}, we have for all $z$ with $\tanh[k_D(z_2,z)]<e_D^\Omega(z_1)-\epsilon-\tanh[k_D(z_1,z_2)]\}$ that
$$\tanh[k_D(z_1,z)]\le \tanh[k_D(z_2,z)]+\tanh[k_D(z_1,z_2)]<e_D^\Omega(z_1)-\epsilon.$$
Thus,
$$B_D^k(z_2,\arctanh[e_D^\Omega(z_1)-\epsilon-\tanh[k_D(z_1,z_2)]]) \subset B_D^k(z_1,\arctanh[e_D^\Omega(z_1)-\epsilon])\subset f(\Omega).$$
This implies that $e_D^\Omega(z_2)\ge e_D^\Omega(z_1)-\epsilon-\tanh[k_D(z_2,z_1)]$.

Since $\epsilon$ is arbitrary, we have $e_D^\Omega(z_2)\ge e_D^\Omega(z_1)-\tanh[k_D(z_2,z_1)]$. Similarly, $e_D^\Omega(z_1)\ge e_D^\Omega(z_2)-\tanh[k_D(z_2,z_1)]$. This proves the theorem.
\end{proof}

We say that a sequence of subdomains $\{D_j\}_{j\ge 1}$ of $D$ is a sequence of \textit{exhausting subdomains} if for any compact subset $K\subset D$, there exists $N>0$ such that $K\subset D_j$ for all $j>N$. In this case, we also say that $\{D_j\}_{j\ge 1}$ \textit{exhausts} $D$.

\begin{cor}\label{cek}
Let $\{D_j\}_{j\ge 1}$ be a sequence of exhausting subdomains of $D$. If $\ds \lim_{j\rightarrow \infty}e_{D_j}^\Omega(z)=e_D^\Omega(z)$ for all $z\in D$, then the convergence is uniform on compact subsets of $D$.
\end{cor}
\begin{proof}
Let $K$ be a compact subset of $D$. Then there exists $0<r<1$ such that $\bigcup_{z\in K}\bv^n(z,r)\Subset D$. Hence there exists $N_1>0$ such that $\bigcup_{z\in K}\bv^n(z,r)\subset D_j$ for all $j>N_1$. Fix any $\epsilon>0$ and take $\delta=r\epsilon/3$. Since $\{\bv^n(z,\delta)\}_{z\in K}$ is an open covering of $K$, there is a finite set $\{z_i\}_{i=1}^m$ such that $K\subset \bigcup_{i=1}^m \bv^n(z_i,\delta)$. For any $z\in K$, there is some $z_i$ such that $z\in \bv^n(z_i,\delta)$. By Theorem \ref{tec} and the decreasing property of the Kobayashi distance, we have
\begin{align*}
|e_D^\Omega(z)-e_{D_j}^\Omega(z)|
&\le |e_D^\Omega(z)-e_D^\Omega(z_i)|+|e_D^\Omega(z_i)-e_{D_j}^\Omega(z_i)|+|e_{D_j}^\Omega(z)-e_{D_j}^\Omega(z)|\\
&\le \tanh[k_D(z,z_i)]+|e_D^\Omega(z_i)-e_{D_j}^\Omega(z_i)|+\tanh[k_{D_j}(z,z_i)]\\
&\le 2\tanh[k_{\bv^n(z_i,r)}(z,z_i)]+|e_D^\Omega(z_i)-e_{D_j}^\Omega(z_i)|\\
&<2\epsilon/3+|e_D^\Omega(z_i)-e_{D_j}^\Omega(z_i)|
\end{align*}
On the other hand, there exists $N_2>0$ such that $|e_D^\Omega(z_i)-e_{D_j}^\Omega(z_i)|<\epsilon/3$ for all $z_i$ and $j>N_2$. Take $N=\max\{N_1,N_2\}$. Then for any $j>N$, we have $|e_D^\Omega(z)-e_{D_j}^\Omega(z)|<\epsilon$ for all $z\in K$.
This completes the proof.
\end{proof}

The condition $\ds \lim_{j\rightarrow \infty}e_{D_j}^\Omega(z)=e_D^\Omega(z)$ in the previous corollary is usually referred to as the \textit{stability} of the Fridman invariant, which was shown to be true when $D$ is Kobayashi complete in \cite[Theorem 2.1]{Fridman1983}. Under the weaker assumption of $D$ being taut (or bounded), we have the following inequality.

\begin{thm}\label{tes}
Suppose that $D$ is bounded or taut. Let $\{D_j\}_{j\ge 1}$ be a sequence of exhausting subdomains of $D$. Then for any $z\in D$, $\ds \limsup_{j\rightarrow \infty} e_{D_j}^\Omega(z)\le e_D^\Omega(z)$.
\end{thm}

To prove Theorem \ref{tes}, we need the following

\begin{lem}\label{lbe}
Let $\{D_j\}_{j\ge 1}$ be a sequence of exhausting subdomains of $D$. Then for any $z\in D$ and $r>0$, $\{B_{D_j}^k(z,r)\}_{j\ge 1}$ exhausts $B_D^k(z,r)$.
\end{lem}
\begin{proof}
By Lemma \ref{lbc}, we know that $B_D^k(z,r)$ is a subdomain of $D$ for any $z\in D$ and $r>0$. Firstly, we show that
$$\lim_{j\rightarrow \infty}k_{D_j}(z',z'')=k_D(z',z''),\ \ \forall z',z''\in D.$$
Consider a sequence of subdomains $\{G_j\}_{j\ge 1}$ such that (i) $G_j\Subset D$, (ii) $G_j\subset G_{j+1}$, (iii) $D=\bigcup_{j\ge 1} G_j$. By \cite[Proposition 3.3.5]{Pflug2013}, we have
$$\lim_{j\rightarrow \infty}k_{G_j}(z',z'')=k_D(z',z''),\ \ \forall z',z''\in D.$$
For any $j\ge 1$, there exists $N_j>0$ such that $G_j\subset D_i$, for all $i>N_j$. By the decreasing property of the Kobayashi distance, we get
$$\lim_{j\rightarrow \infty}k_{D_j}(z',z'')=k_D(z',z''),\ \ \forall z',z''\in D.$$

Now we prove that for any $K\Subset B_D^k(z,r)$, there exists $N>0$ such that $K\subset B_{D_j}^k(z,r)$ for all $j>N$.

Since $k_D(z,\cdot)$ is continuous, there exists $0<r_0<r$ such that $k_D(z,w)\le r_0$ for all $w\in K$. To show that $K\subset B^k_{D_j}(z,r)$, we need to check that $k_{D_j}(z,w)<r$ for all $w\in K$. Since $K$ is a compact subset of $B_D^k(z,r)$, there exists $\delta>0$ such that $\bigcup_{w\in K} \bv^n(w,\delta)\Subset B_D^k(z,r)$. Hence, there exists $N_1>0$ such that $\bigcup_{w\in K} \bv^n(w,\delta)\subset D_j$ for all $j>N_1$.

Let $0<\epsilon<r-r_0$ and take $\delta_1=\delta \tanh(\epsilon/3)$. Since $\{\bv^n(z,\delta_1)\}_{z\in K}$ is an open covering of $K$, there is a finite set $\{z_i\}_{i=1}^m$ such that $K\subset \bigcup_{i=1}^m \bv^n(z_i,\delta_1)$. It is clear that there exists $N_2>0$ such that $|k_{D_j}(z,z_l)-k_D(z,z_l)|<\epsilon/3$ for any $j>N_2$ and $1\le l\le m$. For any $w\in K$, there is some $z_l$ such that $w\in \bv^n(z_l,\delta_1)$. Set $N=\max\{N_1,N_2\}$. Then for all $j>N$, by the decreasing property of the Kobayashi distance, we have
\begin{align*}
&|k_{D_j}(z,w)-k_D(z,w)|\\
\le &|k_{D_j}(z,w)-k_{D_j}(z,z_l)|+|k_{D_j}(z,z_l)-k_D(z,z_l)|+|k_D(z,z_l)-k_D(z,w)|\\
\le &k_{D_j}(z_l,w)+|k_{D_j}(z,z_l)-k_D(z,z_l)|+k_D(z_l,w)\\
\le &2k_{\bv^n(z_l,\delta)}(z_l,w)+|k_{D_j}(z,z_l)-k_D(z,z_l)|\\
<&2\epsilon/3+\epsilon/3=\epsilon.
\end{align*}
Therefore, $k_{D_j}(z,w)<k_D(z,w)+\epsilon\le r_0+\epsilon<r$ for all $w\in K$ and $j>N$. This completes the proof.
\end{proof}

\begin{proof}[Proof of Theorem \ref{tes}]
Since the proof for the taut case is similar as (and simpler than) for the bounded case, we will assume that $D$ is bounded.

For any $z\in D$, let $e_{D_{l_i}}^\Omega$ be a sequence such that $\ds \lim_{l_i\rightarrow \infty}e_{D_{l_i}}^\Omega(z)=\limsup_{j\rightarrow \infty} e^\Omega_{D_j}(z)=:\tanh r$. For any $0<\epsilon<r$, there exists $N_1>0$ such that $e_{D_{l_i}}^\Omega>\tanh(r-\epsilon)$ for all $l_i>N_1$.

Without loss of generality, assume that $0\in \Omega$. By definition, for any $l_i>N_1$, there exists an open holomorphic embedding $f_{l_i}:\Omega \rightarrow D_{l_i}$ such that $f_{l_i}(0)=z$ and $B_{D_{l_i}}^k(z,r-\epsilon) \subset f_{l_i}(\Omega)$. Since $D$ is bounded, by Montel's theorem, there exists a subsequence $\{f_{k_i}\}$ of $\{f_{l_i}\}$ which converges to a holomorphic map $f:\Omega \rightarrow \bar{D}$ uniformly on compact subsets of $\Omega$.

By Lemma \ref{lbc}, each $B_{D_{l_i}}^k(z,r-\epsilon)$ is a domain. Define $g_{l_i}=f_{l_i}^{-1}|B_{D_{l_i}}^k(z,r-\epsilon)$. By Montel's theorem and Lemma \ref{lbe}, we may assume that the sequence $g_{k_i}$ converges uniformly on compact subsets of $B_D^k(z,r-\epsilon)$ to a holomorphic map $g:B_D^k(z,r-\epsilon)\rightarrow \bar\Omega$.

Take $s>0$ such that $\bv^n(z,s)\Subset B_D^k(z,r-\epsilon)$. By Lemma \ref{lbe}, there exists $N>N_1$ such that $\bv^n(z,s)\subset B_{D_{l_i}}^k(z,r-\epsilon)$, for all $l_i>N$. Consider $g_{l_i}|\bv^n(z,s)$. By Cauchy's inequality, $|\det g_{l_i}'(z)|<c$ for all $l_i>N$, for some positive constant $c$. So we have $|\det f_{l_i}'(0)|>\frac{1}{c}$ for all $l_i>N$. Thus, we have $|\det f'(0))|>0$ and $|\det g'(z)|>0$. By Lemma\ref{lh}, both $f$ and $g$ are injective. In particular, $f(\Omega)\subset D$ with $f(0)=z$ and $g(B_D^k(z,r-\epsilon))\subset \Omega$ with $g(z)=0$. Since $f\circ g(w)=w$ for all $w\in B_D^k(z,r-\epsilon)$, we get $e_D^\Omega(z)\ge \tanh(r-\epsilon)$. Since $\epsilon$ is arbitrary ,we have $\ds e_D^\Omega(z)\ge \tanh r=\limsup_{j\rightarrow \infty} e^\Omega_{D_j}(z)$.
\end{proof}

Based on Corollary \ref{cek} and Theorem \ref{tes}, we can slightly refine \cite[Theorem2.1]{Fridman1983} as follows.

\begin{thm}
Suppose that $D$ is Kobayashi complete and $\{D_j\}_{j\ge 1}$ exhausts $D$. Then $\ds \lim_{j\rightarrow \infty} e_{D_j}^\Omega(z)=e_D^\Omega(z)$ uniformly on compact subsets of $D$.
\end{thm}
\begin{proof}
Since $D$ is Kobayashi complete, thus taut, we have $\ds \limsup_{j\rightarrow \infty} e_{D_j}^\Omega(z)\le e_D^\Omega(z)$ for all $z\in D$, by Theorem\ref{tes}.

For $z\in D$ and $0<\epsilon<e_D^\Omega(z)$, by the definition of Fridman invariant and the completeness of $D$, there exists an open holomorphic embedding $f:\Omega \rightarrow D$ such that $B_D^k(z,e_D^\Omega(z)-\epsilon)\Subset f(\Omega)$. Thus, there exists $\delta>0$ such that $B_D^k(z,e_D^\Omega(z)-\epsilon)\subset f((1-\delta)\Omega)\Subset D$. Hence, there exists $N>0$ such that $B_D^k(z,e_D^\Omega(z)-\epsilon)\subset f((1-\delta)\Omega)\subset D_j$ for all $j>N$. By the decreasing property of the Kobayashi distance, we have $B^k_{D_j}(z,e_D^\Omega(z)-\epsilon)\subset B_D^k(z,e_D^\Omega(z)-\epsilon)$. So we have $B^k_{D_j}(z,e_D^\Omega(z)-\epsilon)\subset f((1-\delta)\Omega)$ for all $j>N$, which implies that $\ds \liminf_{j\rightarrow \infty} e_{D_j}^\Omega(z)\ge e_D^\Omega(z)-\epsilon$. Since $\epsilon$ is arbitrary, we get $\ds \liminf_{j\rightarrow \infty} e_{D_j}^\Omega(z)\ge e_D^\Omega(z)$ and hence $\ds \lim_{j\rightarrow \infty}e_{D_j}^\Omega(z)=e_D^\Omega(z)$. By Corollary \ref{cek}, the convergence is uniform on compact subsets of $D$.
\end{proof}

\section{Generalized squeezing functions}\label{S:squeezing}

Throughout this section, we suppose that $D$ is a bounded domain in $\cv^n$ and $\Omega$ is a bounded, balanced and convex domain in $\cv^n$ (unless otherwise stated).

Denote by $k_\Omega$ and $c_\Omega$ the Kobayashi and Carath\'{e}odory distance on $\Omega$, respectively. The following Lempert's theorem is well-known:

\begin{thm}\cite[Theorem 1]{L:convex}\label{T:convex}
On a convex domain $\Omega$, $k_\Omega=c_\Omega$.
\end{thm}

Combining Theorem \ref{T:convex} with \cite[Proposition 2.3.1 (c)]{Pflug2013}, we have the following key lemma.

\begin{lem}\label{lnk}
For any $z\in \Omega$, $\rho_\Omega(z)=\tanh (k_\Omega(0,z))=\tanh (c_\Omega(0,z))$.
\end{lem}

We will also need the following basic fact.

\begin{lem}\label{lmn}
$\rho_\Omega$ is a $\cv$-norm.
\end{lem}
\begin{proof}
For any $z_1$, $z_2\in \cv^n$, we want to show that $\rho_\Omega(z_1+z_2)\le \rho_\Omega(z_1)+\rho_\Omega(z_2)$.

Fix $\epsilon >0$. Take $c_1=\rho_\Omega(z_1)+\epsilon/2$ and $c_2=\rho_\Omega(z_2)+\epsilon/2$, Then $z_1/c_1 \in \Omega$ and $z_2/c_2 \in \Omega$. Since $\Omega$ is convex, we get
$$\frac{z_1+z_2}{c_1+c_2}=\frac{c_1}{c_1+c_2}\frac{z_1}{c_1}+\frac{c_2}{c_1+c_2}\frac{z_2}{c_2}\in \Omega$$
Hence, $\rho_\Omega( z_1+z_2)\le c_1+c_2 \le \rho_\Omega(z_1)+ \rho_\Omega(z_2)+\epsilon$. Since $\epsilon$ is arbitrary, we obtain $\rho_\Omega( z_1+z_2)\le \rho_\Omega(z_1)+ \rho_\Omega(z_2)$.

Since $\Omega$ is bounded, it is obvious that $\rho_\Omega(z)>0$ for all $z\neq 0$, which completes the proof.
\end{proof}

We say that $f\in \oc_u(D,\Omega)$ is an \textit{extremal map} at $z\in D$ if $\Omega(s_D^\Omega(z))\subset f(D)$. When $\Omega=\bv^n$, the existence of extremal maps was given in \cite[Theorem 2.1]{Deng2012}. The proof of the next theorem is very similar to that of Theorem \ref{tee} and \cite[Theorem 2.1]{Deng2012}, based on Montel's theorem and the generalized Hurwitz theorem, so we omit the details.

\begin{thm}\label{tse}
An extremal map exists at each $z\in D$.
\end{thm}

As an immediate corollary, we have

\begin{cor}
$s_D^\Omega(z)=1$ for some $z\in D$ if and only if $D$ is biholomorphically equivalent to $\Omega$.
\end{cor}

In \cite[Theorem 3.1]{Deng2012}, it was shown that $s_D(z)$ is continuous. Moreover, it was given in \cite[Theorem 3.2]{Deng2012} without details the following inequality:
$$|s_D(z_1)-s_D(z_2)|\le 2\tanh[k_D(z_1,z_2)],\ \ \ z_1,z_2\in D.$$
Our next theorem gives the same inequality for generalized squeezing functions, and in particular shows that they are also continuous.

\begin{thm}\label{tsc}
For any $z_1,z_2\in D$, we have
$$|s_D^\Omega(z_1)-s_D^\Omega(z_2)|\le 2\tanh[k_D(z_1,z_2)].$$
In particular, $s_D^\Omega(z)$ is continuous.
\end{thm}
\begin{proof}
By Theorem \ref{tse}, there exists a holomorphic embedding $f:D \rightarrow \Omega$ such that $f(z_1)=0$ and $\Omega(s_D^\Omega(z_1)) \subset f(D)$.

If $\tanh[k_D(z_1,z_2)]\ge s_D^\Omega(z_1)$, then it is obvious that
$$s_D^\Omega(z_2)>0\ge \frac{s_D^\Omega(z_1)-\tanh[k_D(z_1,z_2)]}{1+\tanh[k_D(z_1,z_2)]}.$$

Suppose now that $\tanh[k_D(z_1,z_2)]<s_D^\Omega(z_1)$. By the decreasing property of the Kobayashi distance and Lemma \ref{lnk}, we have
$$\begin{aligned}
s_D^\Omega(z_1)&>\tanh[k_D(z_1,z_2)]=\tanh[k_{f(D)}(f(z_1),f(z_2))]\\
&\ge \tanh[k_\Omega(f(z_1),f(z_2))]=\tanh[k_\Omega(0,f(z_2))]=\rho_\Omega(f(z_2)).
\end{aligned}$$
Define
$$h(w):=\frac{w-f(z_2)}{1+\tanh[k_D(z_1,z_2)]},$$
and set $g(z)=h\circ f(z)$. Then $g\in \oc_u(D,\Omega)$ and $g(z_2)=0$.

For any $w\in \Omega$ with
$$\rho_\Omega(w)<\frac{s_D^\Omega(z_1)-\tanh[k_D(z_1,z_2)]}{1+\tanh[k_D(z_1,z_2)]},$$
we have 
$$\rho_\Omega(h^{-1}(w)-f(z_2))=\rho_\Omega(h^{-1}(w)-h^{-1}(g(z_2)))<s_D^\Omega(z_1)-\tanh[k_D(z_1,z_2)].$$
Since $\rho _\Omega(z)$ is a $\cv$-norm by Lemma \ref{lmn}, we get
\begin{align*}
\rho_\Omega(h^{-1}(w))=\rho_\Omega(h^{-1}(w)-f(z_1))
& \le \rho_\Omega(h^{-1}(w)-f(z_2))+\rho_\Omega(f(z_2)-f(z_1))\\
&< s_D^\Omega(z_1)-\tanh[k_D(z_1,z_2)]+\rho_\Omega(f(z_2))\\
&\le s_D^\Omega(z_1).
\end{align*}
This implies that
$$\Omega\left(\frac{s_D^\Omega(z_1)-\tanh[k_D(z_1,z_2)]}{1+\tanh[k_D(z_1,z_2)]}\right)\subset h(\Omega(s_D^\Omega(z_1)))\subset g(D).$$
So we have
$$s_D^\Omega(z_2)\ge \frac{s_D^\Omega(z_1)-\tanh[k_D(z_1,z_2)]}{1+\tanh[k_D(z_1,z_2)]}.$$
Hence,
$$s_D^\Omega(z_1)\le s_D^\Omega(z_2)+(s_D^\Omega(z_2)+1)\tanh[k_D(z_1,z_2)]\le s_D^\Omega(z_2)+2\tanh[k_D(z_1,z_2)].$$
Similarly,
$$s_D^\Omega(z_2)\le s_D^\Omega(z_1)+2\tanh[k_D(z_1,z_2)].$$
Therefore, $|s_D^\Omega(z_1)-s_D^\Omega(z_2)|\le 2\tanh[k_D(z_1,z_2)]$ for all $z_1,z_2\in D$.

Since the Kobayashi distance is continuous (see e.g. \cite{Pflug2013}), we get that $s_D^\Omega(z)$ is continuous.
\end{proof}

In case that $\Omega$ is homogeneous, we have better estimates as follows.
 
\begin{thm}\label{tscb}
If $\Omega$ is bounded, balanced, convex and homogeneous, then for any $z_1,z_2\in D$, we have
$$|s_D^\Omega(z_1)-s_D^\Omega(z_2)|\le \tanh[k_D(z_1,z_2)].$$
\end{thm}
\begin{proof}
By Theorem \ref{tse}, there exists a holomorphic embedding $f:D \rightarrow \Omega$ such that $f(z_1)=0$ and $\Omega(s_D^\Omega(z_1)) \subset f(D)$.

If $\tanh[k_D(z_1,z_2)]\ge s_D^\Omega(z_1)$, then it is obvious that
$$s_D^\Omega(z_2)>0\ge s_D^\Omega(z_1)-\tanh[k_D(z_1,z_2)].$$

Suppose now that $\tanh[k_D(z_1,z_2)]<s_D^\Omega(z_1)$. Since $\Omega$ is homogeneous, there exists $\psi \in \aut(\Omega)$ such that $\psi \circ f(z_2)=0$.

For any $w\in \Omega$ with
$$\tanh[k_{\psi(\Omega)}(w,\psi \circ f(z_2))]<s_D^\Omega(z_1)-\tanh[k_D(z_1,z_2)],$$
by the decreasing property of the Kobayashi distance and Lemma \ref{ltanh}, we have
\begin{align*}
\tanh[k_\Omega(\psi^{-1}(w),f(z_1))]
& \le \tanh[k_\Omega(\psi^{-1}(w),f(z_2))]+\tanh[k_\Omega(f(z_2),f(z_1))]\\
& \le \tanh[k_{\psi(\Omega)}(w,\psi \circ f(z_2))]+\tanh[k_{f(D)}(f(z_2),f(z_1))]\\
& < s_D^\Omega(z_1)-\tanh[k_D(z_1,z_2)]+\tanh[k_D(z_1,z_2)]\\
& =s_D^\Omega(z_1).
\end{align*}
By Lemma \ref{lnk}, this implies that $\psi^{-1}(w)\in \Omega(s_D^\Omega(z_1))$. Hence,
$$\{w:\tanh[k_{\psi(\Omega)}(w,\psi \circ f(z_2))]<s_D^\Omega(z_1)-\tanh[k_D(z_1,z_2)]\} \subset \psi(\Omega(s_D^\Omega(z_1))).$$ 
Since $\psi \circ f(z_2)=0$, again by Lemma \ref{lnk}, we have
$$\Omega(s_D^\Omega(z_1)-\tanh[k_D(z_1,z_2)])\subset \psi(\Omega(s_D^\Omega(z_1)))\subset \psi\circ f(D).$$
Thus,
$$s_D^\Omega(z_2)\ge s_D^\Omega(z_1)-\tanh[k_D(z_1,z_2)].$$
Similarly,
$$s_D^\Omega(z_1)\ge s_D^\Omega(z_2)-\tanh[k_D(z_1,z_2)].$$
Therefore, $|s_D^\Omega(z_1)-s_D^\Omega(z_2)|\le \tanh[k_D(z_1,z_2)]$ for all $z_1,z_2\in D$.
\end{proof}

The following stability result generalizes and slightly refines \cite[Theorem2.1]{Deng2016}.

\begin{thm}
If $\{D_l\}_{l\ge 1}$ exhausts $D$, then $\ds \lim_{l\rightarrow \infty} s_{D_l}^\Omega(z)=s_D^\Omega(z)$ uniformly on compact subsets of $D$.
\end{thm}
\begin{proof}
Since the proof of the convergence is similar to that of \cite[Theorem2.1]{Deng2016}, we omit the details.

Now suppose that the convergence is not uniform on compact subsets of $D$. Then, there exist a compact subset $K\subset D$, $\epsilon>0$, a subsequence $\{l_j\}$ and $z_{l_j}\in K\subset D_{l_j}$ such that
$$|s^\Omega_{D_{l_j}}(z_{l_j})-s_D^\Omega(z_{l_j})|\ge \epsilon.$$
Since $K$ is compact, there exists a convergent subsequence, again denoted by $\{z_{l_j}\}$, with $\lim_{j\rightarrow \infty} z_{l_j}=z\in K$. Choose $r>0$ such that $\overline{\bv^n(z,r)}\subset D$. Then, there is $N_1>0$ such that $z_{l_j}\in \bv^n(z,r)\subset D_{l_j}$ for all $l_j>N_1$. By Theorem \ref{tsc} and the decreasing property of the Kobayashi distance, for all $l_j>N_1$ we have
\begin{align*}
|s^\Omega_{D_{l_j}}(z_{l_j})-s_D^\Omega(z_{l_j})|
& \le |s^\Omega_{D_{l_j}}(z_{l_j})-s^\Omega_{D_{l_j}}(z)|+|s^\Omega_{D_{l_j}}(z)-s_D^\Omega(z)|+|s_D^\Omega(z)-s_D^\Omega(z_{l_j})|\\
&\le 2\tanh[k_{D_{l_j}}(z_{l_j},z)]+|s^\Omega_{D_{l_j}}(z)-s_D^\Omega(z)|+2\tanh[k_D(z,z_{l_j})]\\
& \le 4\tanh\left(\frac{\|z_{l_j}-z\|}{r}\right)+|s^\Omega_{D_{l_j}}(z)-s_D^\Omega(z)|.
\end{align*}
It is clear that there is $N_2>0$ such that for all $l_j>N_2$ we have
$$\tanh\left(\frac{\|z_{l_j}-z\|}{r}\right)<\frac{\epsilon}{6}\ \ \textup{and}\ \ |s^\Omega_{D_{l_j}}(z)-s_D^\Omega(z)|<\frac{\epsilon}{3}.$$
Set $N=\max\{N_1,N_2\}$. Then for all $l_j>N$ we have
$$|s^\Omega_{D_{l_j}}(z_{l_j})-s_D^\Omega(z_{l_j})|<\epsilon,$$
which is a contradiction.
\end{proof}

The notion of the squeezing function was originally introduced to study the ``uniform squeezing" property. In this regard, we have the following

\begin{thm}\label{te}
For two bounded, balanced and convex domains $\Omega_1$ and $\Omega_2$ in $\cv^n$, $s^{\Omega_1}_D(z)$ has a positive lower bound if and only if $s^{\Omega_2}_D(z)$ has a positive lower bound.
\end{thm}
\begin{proof}
It suffices to prove the equivalence when $\Omega_2=\bv^n$. By Lemma \ref{lmn}, $\rho_{\Omega_1}(z)$ is a $\cv$-norm. Thus, it is continuous and there exist $M\ge m>0$ such that $m\|z\|\le \rho_{\Omega_1}(z) \le M\|z\|$. Then, one readily checks using the definition that
$$\frac{s^{\Omega_1}_D(z)}{M}\le s^{\bv_n}_D(z)\le \frac{s^{\Omega_1}_D(z)}{m}.$$
\end{proof}

Combining Theorem \ref{te} with \cite[Theorems 4.5 \& 4.7]{Deng2012}, we have the following

\begin{thm}\label{tckc}
If $s_D^\Omega(z)$ has a positive lower bound, then $D$ is complete with respect to the Carath\'{e}odory distance, the Kobayashi distance and the Bergman distance of $D$.
\end{thm}

\section{Comparison of Fridman invariants and generalized squeezing functions}\label{S:comparison}

Since Fridman invariants and generalized squeezing functions are similar in spirit to the Kobayashi-Eisenman volume form $K_D$ and the Carath\'{e}odory volume form $C_D$, respectively, it is natural to study the comparison of them. For this purpose, we will always assume that $D$ is a bounded domain in $\cv^n$ and $\Omega$ is a bounded, balanced, convex and homogeneous domain in $\cv^n$.

Similar to the classical quotient invariant $M_D(z):=C_D(z)/K_D(z)$, we introduce the quotient $m_D^\Omega(z)=s_D^\Omega(z)/e_D^\Omega(z)$, which is also a biholomorphic invariant. When $\Omega=\bv^n$, we simply write $m_D(z)=s_D(z)/e_D(z)$.

In \cite{Nikolov-Verma2018}, Nikolov and Verma have shown that $m_D(z)$ is always less than or equal to one. The next result shows that the same is true for $m_D^\Omega(z)$.

\begin{thm}\label{te>s}
For any $z\in D$, we have $m_D^\Omega(z)\le 1$.
\end{thm}
\begin{proof}
For any $z\in D$, by Theorem \ref{tse}, there exists a holomorphic embedding $f:D \rightarrow \Omega$ such that $f(z)=0$ and $\Omega(s_D^\Omega(z))\subset f(D)$.

Define $g(w):=f^{-1}(s_D^\Omega(z)w)$, which is an injective holomorphic mapping from $\Omega$ to $D$ with $g(0)=z$. By the decreasing property of the Kobayashi distance and Lemma \ref{lnk}, we have
$$B^k_{f(D)}(0,\arctanh[s_D^\Omega(z)])\subset B^k_\Omega(0,\arctanh[s_D^\Omega(z)])=\Omega(s_D^\Omega(z)).$$
Thus,
$$B_D^k(z,\arctanh[s_D^\Omega(z)])=f^{-1}(B^k_{f(D)}(z,\arctanh[s_D^\Omega(z)]))\subset f^{-1}(\Omega(s_D^\Omega(z)))=g(\Omega).$$
This implies that $e_D^\Omega(z)\ge s_D^\Omega(z)$, i.e. $m_D^\Omega(z)\le 1$.
\end{proof}

A classical result of Bun Wong (\cite[Theorem E]{W:ball}) says that if there is a point $z\in D$ such that $M_D(z)=1$, then $D$ is biholomorphic to the unit ball $\bv^n$. In \cite[Theorem 3]{RY:comparison}, we showed that an analogous result for $m_D(z)$ does not hold. The next result is a generalized version of \cite[Theorem 3]{RY:comparison} for $m_D^\Omega(z)$.

\begin{thm}\label{te=s}
If $D$ is bounded, balanced and convex, then $m_D^\Omega(0)=1$.
\end{thm}
\begin{proof}
By Theorem \ref{tee}, there exists a holomorphic embedding $f:\Omega \rightarrow D$ such that $f(0)=0$ and $B_D^k(0,e_D^\Omega(0))\subset f(\Omega)$.

Define $g(w):=f^{-1}(e_D^\Omega(0)w)$, which is an injective holomorphic mapping from $D$ to $\Omega$ with $g(0)=0$. By the decreasing property of the Kobayashi distance and Lemma \ref{lnk}, we have
$$B_{f(\Omega)}^k(0,\arctanh[e_D^\Omega(0)])\subset B_D^k(0,\arctanh[e_D^\Omega(0)])=D(e_D^\Omega(0)).$$
Thus,
$$B_\Omega^k(0,\arctanh[e_D^\Omega(0)])=f^{-1}(B_{f(\Omega)}^k(z,\arctanh[e_D^\Omega(0)]))\subset f^{-1}(D(e_D^\Omega(0)))=g(\Omega).$$
This implies that $s_D^\Omega(0)\ge e_D^\Omega(0)$. By Theorem \ref{te>s}, we always have $s_D^\Omega(0)\le e_D^\Omega(0)$. This completes the proof.
\end{proof}

\begin{cor}\label{cs=s}
Let $\Omega_i$, $i=1,2$, be two bounded, balanced, convex and homogeneous domains in $\cv^n$. Then $s^{\Omega_2}_{\Omega_1}(z_1)=s^{\Omega_1}_{\Omega_2}(z_2)$ for all $z_1\in \Omega_1$ and $z_2\in \Omega_2$.
\end{cor}
\begin{proof}
Since both $\Omega_1$ and $\Omega_2$ are homogeneous, it suffices to show that $s^{\Omega_2}_{\Omega_1}(0)=s^{\Omega_1}_{\Omega_2}(0)$.

By Lemma \ref{lnk}, we have $B^k_{\Omega_2}(0,\arctanh(r))=\Omega_2(r)$ for $r>0$. Then, by definition, $s^{\Omega_2}_{\Omega_1}(0)=e^{\Omega_1}_{\Omega_2}(0)$. By Theorem \ref{te>s}, we get $s^{\Omega_2}_{\Omega_1}(0)=s^{\Omega_1}_{\Omega_2}(0)$.
\end{proof}

We can also compare generalized squeezing functions for different model domains as follows.

\begin{thm}
Let $\Omega_i$, $i=1,2$, be two bounded, balanced, convex and homogeneous domains in $\cv^n$. Then, for any $z\in D$, we have
$$s^{\Omega_1}_{\Omega_2}(0)s^{\Omega_2}_D(z)\le s^{\Omega_1}_D(z)\le \frac{1}{s^{\Omega_1}_{\Omega_2}(0)}s^{\Omega_2}_D(z).$$
\end{thm}
\begin{proof}
For any $z\in D$, by Theorem \ref{tse}, there exists a holomorphic embedding $f:D \rightarrow \Omega_1$ such that $f(z)=0$ and $\Omega_1(s_D^\Omega(z))\subset f(D)$. And there exists a holomorphic embedding $g:\Omega_1 \rightarrow \Omega_2$ such that $g(0)=0$ and $\Omega_2(s^{\Omega_2}_{\Omega_1}(0))\subset g(\Omega_1)$.

Set $F=g\circ f$. Then $F\in \oc_u(D,\Omega_2)$ with $F(z)=0$. Denote $\Omega=\Omega_2(s^{\Omega_2}_{\Omega_1}(0))$. Then $\Omega$ is a bounded, balanced and convex domain with $\rho _\Omega=\frac{1}{s^{\Omega_2}_{\Omega_1}(0)}\rho _{\Omega_2}$. By the  decreasing property of the Kobayashi distance and Lemma \ref{lnk}, we have
\begin{align*}
B^k_\Omega(0,\arctanh[s^{\Omega_1}_D(z)])&\subset B^k_{g(\Omega_1)}(0,\arctanh[s^{\Omega_1}_D(z)])=g(B^k_{\Omega_1}(0,\arctanh[s^{\Omega_1}_D(z)]))\\
&=g(\Omega_1(s^{\Omega_1}_D(z)))\subset g(f(D))=F(D).
\end{align*}
On the other hand, by Lemma \ref{lnk}, we have
\begin{align*}
B^k_\Omega(0,\arctanh[s^{\Omega_1}_D(z)])&=\{w\in \Omega:\rho _\Omega(w)<s^{\Omega_1}_D(z)\}\\
&=\{w\in \Omega_2:\rho _{\Omega_2}(w)<s^{\Omega_2}_{\Omega_1}(0)s^{\Omega_1}_D(z)\}.
\end{align*}
This implies that $s^{\Omega_2}_D(z)\ge s^{\Omega_2}_{\Omega_1}(0)s^{\Omega_1}_D(z)$. Similarly, $s^{\Omega_1}_D(z)\ge s^{\Omega_1}_{\Omega_2}(0)s^{\Omega_2}_D(z)$. By Corollary \ref{cs=s}, we get
$$s^{\Omega_1}_{\Omega_2}(0)s^{\Omega_2}_D(z)\le s^{\Omega_1}_D(z)\le \frac{1}{s^{\Omega_1}_{\Omega_2}(0)}s^{\Omega_2}_D(z).$$
\end{proof}

We finish our study by computing explicitly some generalized squeezing functions in the next result, which generalizes \cite[Corollary 7.3]{Deng2012}.

\begin{thm}
For any $z\in\Omega\backslash \{0\}$, $s_{\Omega\backslash \{0\}}^\Omega(z)=\rho _\Omega(z)$.
\end{thm}
\begin{proof}
Since $\Omega$ is homogeneous, for any $z\in\Omega\backslash \{0\}$, there exists $\psi\in \aut(\Omega)$ such that $\psi(z)=0$. Then, by Lemma \ref{lnk},
$$\rho_\Omega(\psi(0))=\tanh[k_\Omega(\psi(0),0)]=\tanh[k_\Omega(\psi(0),\psi(z))]=\tanh[k_\Omega(0,z)]=\rho_\Omega(z).$$
It follows that $s^\Omega_{\Omega \backslash \{0\}}(z)\ge \rho _\Omega(z)$.

Next, we show that $s^\Omega_{\Omega \backslash \{0\}}(z)\le \rho _\Omega(z)$. By Theorem \ref{tse}, there exists a holomorphic embedding $f:\Omega\backslash \{0\} \rightarrow \Omega$ such that $f(z)=0$ and $\Omega(s^\Omega_{\Omega\backslash \{0\}}(z))\subset f(\Omega\backslash \{0\})$.

Define $g(w):=f^{-1}(s^\Omega_{\Omega \backslash \{0\}}(z)w)$, which is an injective holomorphic mapping from $\Omega$ to $\Omega\backslash \{0\}$ with $g(0)=z$. By the decreasing property of the Carath\'{e}odory distance and Lemma \ref{lnk}, we have
\begin{align*}
B^{c}_{\Omega \backslash \{0\}}(z,\arctanh[s^\Omega_{\Omega \backslash \{0\}}(z)])
&=f^{-1} (B^{c}_{f(\Omega \backslash \{0\})}(0,\arctanh[s^\Omega_{\Omega \backslash \{0\}}(z)]))\\
&\subset f^{-1} (B^{c}_\Omega(0,\arctanh[s^\Omega_{\Omega \backslash \{0\}}(z)]))\\
&=f^{-1}(\Omega(s^\Omega_{\Omega \backslash \{0\}}(z)))=g(\Omega).
\end{align*}
By Riemann's removable singularity theorem, we have $c_{\Omega \backslash \{0\}}(z_1,z_2)=c_\Omega(z_1,z_2)$ for all $z_1, z_2\in \Omega \backslash \{0\}$. Thus,
$$\arctanh(s^\Omega_{\Omega \backslash \{0\}}(z))\le c_\Omega(z,0)=\arctanh(\rho _\Omega(z)).$$
Hence, $s^\Omega_{\Omega \backslash \{0\}}(z)\le \rho _\Omega(z)$, which completes the proof.
\end{proof}


\begin{thebibliography}{}

\bibitem{Deng2012}
F. Deng, Q. Guan, L. Zhang;
\emph{Some properties of squeezing functions on bounded domains},
Pacific J. Math. {\bf 257} (2012), 319-341.

\bibitem{Deng2016}
F. Deng, Q. Guan, L. Zhang;
\emph{Properties of squeezing functions and global transformations of bounded domains},
Trans. Amer. Math. Soc. {\bf 368} (2016), 2679-2696.

\bibitem{Dengsurvey2019}
F. Deng, Z. Wang, L. Zhang, X. Zhou;
\emph{Holomorphic invariants of bounded domains},
J. Geom. Anal. {\bf 30} (2020), 1204-1217.

\bibitem{Deng-Zhang2019}
F. Deng, X. Zhang;
\emph{Fridman's invariant, squeezing functions, and exhausting domains},
Acta Math. Sin. (Engl. Ser.) {\bf 35} (2019), 1723-1728.

\bibitem{Fridman1983}
B.L. Fridman;
\emph{Biholomorphic invariants of a hyperbolic manifold and some applications},
Trans. Amer. Math. Soc. {\bf 276} (1983), 685-698.

\bibitem{Pflug2013}
M. Jarnicki, P. Pflug;
Invariant Distances and Metrics in Complex Analysis,
2nd ext. ed., de Gruyter Expos. in Math. {\bf 9}, Walter de Gruyter GmbH, Berlin, 2013.

\bibitem{Kobayashi98}
S. Kobayashi;
Hyperbolic Complex Spaces,
Springer Science  Business Media, 1998.

\bibitem{L:convex}
L. Lempert;
\emph{Holomorphic retracts and intrinsic metrics in convex domains},
Anal. Math. {\bf 8} (1982), 257-261.

\bibitem{Nikolov-Verma2018}
N. Nikolov, K. Verma;
\emph{On the squeezing function and Fridman invariants},
J. Geom. Anal. {\bf 30} (2020), 1218-1225.

\bibitem{RY:comparison}
F. Rong, S. Yang;
\emph{On the comparison of the Fridman invariant and the squeezing function},
Complex Var. Elliptic Equ., to appear.

\bibitem{shijihuai1996}
J. Shi;
Basis of Several Complex Variables (in Chinese),
Higher Education Publishing House, Beijing, 1996.

\bibitem{W:ball}
B. Wong;
\emph{Characterizations of the ball in $\cv^n$ by its automorphism group},
Invent. Math. {\bf 41} (1977), 253-257.

\bibitem{Zhang2017}
L. Zhang;
\emph{Intrinsic derivative, curvature estimates and squeezing function},
Science China Math. {\bf 60} (2017), 1149-1162.

\end{thebibliography}
\end{document}